\documentclass[a4paper,11pt]{article}
\usepackage{amssymb,latexsym,amsmath,amscd,amsthm}
\usepackage{comment}

\newtheorem{thm}{Theorem}[section]
\newtheorem{lema}[thm]{Lemma}

\newtheorem{prop}[thm]{Proposition}

\theoremstyle{definition}

\newtheorem{ex}[thm]{Example}

\newtheorem{claim}[thm]{Claim}

\newtheorem{rk}[thm]{Remark}

\newcommand{\ZZ}{\mathbb{Z}}
\newcommand{\CC}{\mathbb{C}}

\newcommand{\RR}{\mathbb{R}}

\newcommand{\dd}{\mathrm{diag}}
\newcommand{\EE}{\delta}
\newcommand{\poly}{\mathcal{P}}
\newcommand{\var}{\mathcal{V}}
\newcommand{\etalchar}[1]{$^{#1}$}

\newcommand{\Arg}{\mathrm{Arg}}
\newcommand{\Log}{\mathrm{Log}}
\newcommand{\Ll}[2]{\mathrm{Log}_{#1}(M)}

\usepackage{appendix}
\usepackage{color}

\definecolor{pink}{rgb}{1,0,1}

\definecolor{garnet}{RGB}{210,15,30}

\title{An open set of $4\times 4$ embeddable matrices whose principal logarithm is not a Markov generator}
\author{Marta Casanellas, Jes\'us Fern\'andez-S\'anchez, Jordi Roca-Lacostena}

\begin{document}

\maketitle

\begin{abstract}

A Markov matrix is embeddable if it can represent a homogeneous continuous-time Markov process. It is well known that if a Markov matrix has real and pairwise-different eigenvalues, then the embeddability can be determined by checking whether its principal logarithm is a rate matrix or not. The same holds for Markov matrices close enough to the identity matrix or that rule a Markov process subjected to certain restrictions. In this paper we prove that this criterion cannot be generalized and we provide open sets of  Markov matrices that are embeddable and whose principal logarithm is a not a rate matrix.
\end{abstract}

\emph{Keywords}: Markov matrix; Markov generator; embedding problem; rate identifiability

\section{Introduction}


The \emph{embedding problem} \cite{Elfving} for Markov matrices consists on deciding whether a Markov (or stochastic) matrix can be written as the exponential of a \emph{rate matrix} (that is, a real matrix whose rows sum to 0 and has non-negative off-diagonal entries). This is equivalent to deciding whether the substitution process ruled by such Markov matrix can be modeled as a homogeneous continuous-time process or not. We say  that a Markov matrix $M$ is \emph{embeddable} if there exists rate matrix $Q$ such that $M=e^Q$.

Although the embedding problem is about eighty years old, it was only solved completely for $2\times 2$ and $3 \times 3$ matrices until very recently \cite{Kingman,Cuthbert73,Johansen,Carette}. A recent manuscript of the authors fully solves it for $4\times4$ matrices and gives a solution for generic matrices of any size \cite{CFR4x4}. Several partial results had been proved on the way (see \cite{Runnenburg,Kingman,Culver,Cuthbert72, Israel} for example).

One of the most relevant and well known results is that the embeddability of a matrix with \emph{real and different eigenvalues} can be determined by checking whether the principal logarithm is a rate matrix. Indeed, this is a consequence of a result in \cite{Culver} which states that a Markov matrix with pairwise different real eigenvalues has only one real logarithm, its principal logarithm. Moreover, the principal logarithm also characterizes the embeddability of Markov matrices that are ``close'' to the identity matrix or have a large enough determinant (\cite{Cuthbert73}, \cite{Israel}, \cite{CFR4x4}) or transition matrices of certain particular random processes (see \cite{ChenJia}, \cite{JJ}, for example). As a consequence,  the principal logarithm has been used in different settings as a tool to detect whether a Markov process may have a homogeneous continuous-time realization or not (see \cite{verbyla, Geweke}). However, as pointed out above, a priori this characterization could only be used under certain conditions.

The following question arises naturally: 
\emph{Is the embeddability of a generic Markov matrix determined by the principal logarithm?} In this paper, we give a negative answer to this question. We provide non-empty \emph{open} sets of Markov matrices that are embeddable and whose principal logarithm is not a rate matrix.
More precisely, we deform particular examples in order to generate  non-empty open sets of $4\times 4$ Markov matrices (with different eigenvalues) that are embeddable and have a unique Markov generator which, moreover, is different from the principal logarithm (see Theorem 3.3 for the precise statement).


 The organization of this paper is as follows. In section \ref{sec:Preliminaries} we start by properly defining the embedding problem and present some known results regarding matrix logarithms. In section \ref{sec:Obert} we first provide examples of embeddable matrices with non-negative eigenvalues whose principal logarithm is not a rate matrix and use these examples to obtain open subsets in the space of $4\times4$ Markov matrices containing embeddable matrices with the same property. Finally, in Section \ref{sec:GenerateExamples} we give an insight on how these examples were obtained.

\section{Preliminaries} \label{sec:Preliminaries}
In this section we introduce the notation an background needed for the sequel.

We say that a real square matrix $M$ is a \emph{Markov matrix} if its entries are non-negative and all its rows sum to 1. A real square matrix $Q$ is a \emph{rate matrix} if its off-diagonal entries are non-negative and its rows sum to 0. It is known that $Q$ is a rate matrix if and only if $M(t)=e^{Qt}$ is a Markov matrix for all $t\geq 0$. In this case, we say that $M=e^Q$ is \emph{embeddable} (since it can be embedded into a homogeneous continuous-time process) and we say that $Q$ is a \emph{Markov generator} for $M$. The \emph{embedding problem} consists on deciding whether a given Markov matrix is embeddable or not \cite{Elfving}.
Note that the space of $n\times n$ rate matrices and the space of $n\times n$ Markov matrices have the same dimension ($n^2-n$) and, as the exponential map is a local diffeomorphism from rate matrices to embeddable matrices,  the space of embeddable Markov matrices has the same dimension, too.

We use the notation $\log_k(\lambda)$ to denote the $k$-th determination of the logarithm of $\lambda \in \CC$, that is, $\log_k(\lambda)=\log|\lambda|+ (\text{Arg}(\lambda) + 2\pi k) i$ where $\text{Arg}(\lambda) \in (-\pi, \pi]$ is the \emph{principal argument} of $\lambda$. The principal logarithm of $\lambda$,  $\log_0(\lambda)$, will be denoted as $\log(\lambda)$ for ease of reading. We say that a matrix $Q$ is a \emph{logarithm} of a matrix $M$ if $e^Q=M$. From the exponential series of a matrix, $e^Q = \sum_{n\geq 0} \frac{Q^n}{n!}$, we have that, if $v$ is an eigenvector with eigenvalue $\lambda$ of $Q$, then $v$ is an eigenvector with eigenvalue $e^\lambda$ of $e^Q$.
 It is well known that if $M$ has no negative or null eigenvalues then there is a unique logarithm of $M$ whose eigenvalues are the principal logarithm of the eigenvalues of $M$ \cite{Higham}. In this case, we refer to this logarithm as the \emph{principal logarithm of} $M$ and denote it by $\Log(M)$.

 Along this paper we work with $4\times 4$ diagonalizable Markov matrices $M$ with an eigendecomposition $M= P\; \dd(1,\lambda, \mu ,\overline{\mu}) \; P^{-1}$ with $\lambda \in (0,1)$ and $\mu \in \CC\setminus \RR$ such that $\Arg(\mu)>0$. According to \cite{Culver}, all real logarithms with rows summing to zero of such matrices are of the form $\Log_k(M)$ where
\begin{equation}\label{eq:LogKDef}\
\Log_k(M):= P \; \dd(0,\log(\lambda), \log_k(\mu) ,\overline{\log_k(\mu)}) \; P^{-1}
\end{equation}
does not depend on $P$.

The  results in this paper arise from transition matrices of a nucleotide substitution model known as the \emph{Strand Symmetric Model} \cite{SSM}. The matrices in the model are of the following form:
\begin{equation}\label{eq:SSM}
\begin{small}
	\begin{pmatrix}
		a & b & c & d \\
		e & f & g & h \\
		h & g & f & e \\
		d & c & b & a \\
	\end{pmatrix}.
\end{small}
\end{equation}
We will refer to \emph{real} matrices with this structure as \emph{SS matrices}. A straightforward computation shows that the product and sum of SS matrices is closed within the model. Hence, the exponential of a rate SS matrix is a Markov SS matrix. Moreover, according to Theorem 6.7 by \cite{Anya} all the SS matrices can be transformed into a 2-block-diagonal matrix via the same change of basis  and hence, if $M$ is an SS matrix with a conjugated pair of non-real eigenvalues then  $\Log_k(M)$  is also an SS matrix for all $k \in \ZZ$.

\section{Main result}\label{sec:Obert}


In this section we show that, for any $k\in \ZZ$, there is a non-zero measure set of embeddable $4 \times 4$ Markov matrices whose unique Markov generator is its $\Log_k$ (see Theorem \ref{thm:OpenSet}). This is proved by deforming the particular examples below.

\begin{ex}\label{ex:CounterexK} We distinguish two cases depending on the sign of $l$.

\begin{itemize}
		\item If $l \geq 0$, consider the matrix $M=P_+ \; D_+ \; P_+^{-1}$ where
				$$P_+:=\begin{small}
				\begin{pmatrix}
					1 & 6l+2  & 1-i  & 1+i \\
					1 & -2l-1 & -i  & i\\
					1 & -2l-1 &  i  & -i\\
					1 & 6l+2  &  -1+i &  -1-i \\
				\end{pmatrix}
				\end{small}  \quad \textrm{ and } \quad D_+:= \dd (	1 , e^{(1-8l)\pi} , e^{-2\pi(1+2l)} i , -e^{-2\pi(1+2l)} i). $$

			 \noindent A straightforward computation shows that $M$ is a Markov matrix. Further computations show that, for any $k\in \ZZ$, $$\Ll{k}{M}= \frac{\pi}{4}\begin{small}
				\begin{pmatrix}
					-9-20 l-4 k & 6+12 l+8 k & 2+12 l-8 k & 1-4 l+4 k\\
					1+4 l-4 k & -5-12 l+4 k & 1+4 l-4 k & 3+4 l+4 k\\
					3+4 l+4 k & 1+4 l-4 k & -5-12 l+4 k & 1+4 l-4 k\\
					1-4 l+4 k & 2+12 l-8 k & 6+12 l+8 k  & -9-20 l-4 k\\
				\end{pmatrix}
				\end{small}.$$

			Since $k,l\in \ZZ$ and $l\geq 0$, the only possible choice for $k$ so that off-diagonal entries are non-negative is $k=l$ (this can be easily seen by looking at the entries $(2,1)$ and $(4,1)$ for instance). In this case we have: $$\Ll{l}{M}= \frac{\pi}{4}\begin{small}
				\begin{pmatrix}
					-9-24 l & 6+20 l & 2+4 l & 1\\
					1 & -5-8 l & 1 & 3+8 l\\
					3+8 l & 1 & -5-8 l & 1\\
					1 & 2+4 l & 6+20 l & -9-24 l\\
				\end{pmatrix}
				\end{small}.$$
			In particular $M$ is embeddable and $\Log_{l}(M)$ is its unique Markov generator.\\

		\item If $l < 0$, consider the matrix $M=P_- \; D_- \; P_-^{-1}$ where
				$$P_-:=\begin{small}
					\begin{pmatrix}
						1 & 6l+1  & 1-i  & 1+i \\
						1 & 2l & -i  & i\\
						1 & 2l &  i  & -i\\
						1 & 6l+1  &  -1+i &  -1-i \\
				\end{pmatrix}
				\end{small}, \quad \textrm{ and } \quad  D_-:= \dd (	1 , e^{(1+8l)\pi} , e^{4l\pi} i , -e^{4l\pi} i). $$

			As above, $M$ is a Markov matrix and, for any $k\in \ZZ$, $$\Ll{k}{M}= \frac{\pi}{4}\begin{small}
				\begin{pmatrix}
					3+20 l+4 k & -12 l+8 k & -4-12 l-8 k & 1+4 l-4 k\\
					-1-4 l-4 k & -1+12 l-4 k & 1-4 l+4 k & 1-4 l+4 k\\
					1-4 l+4 k & 1-4 l+4 k & -1+12 l-4 k &-1-4 l-4 k \\
					1+4 l-4 k & -4-12 l-8 k & -12 l+8 k & 3+20 l+4 k\\
				\end{pmatrix}
				\end{small}.$$

			Since $k,l\in \ZZ$ and $l<0$, the only possible choice for $k$ that produces only non-negative off-diagonal entries is $k=l$. In this case we have
$$\Ll{l}{M}= \frac{\pi}{4}\begin{small}
				\begin{pmatrix}
					3+24 l  & -4l  & -4-20 l & 1\\
					-1-8 l & -1+8 l & 1 & 1\\
					1 & 1 & -1+8 l & -1-8 l\\
					1 & -4-20 l & -4 l & 3+24 l\\
				\end{pmatrix}
				\end{small}.$$
			In particular, $M$ is embeddable and $\Ll{l}{M}$ is its unique Markov generator.			
	\end{itemize}
\end{ex}

Note that for $l\neq 0$, the matrices described in the previous example are embeddable matrices whose principal logarithm is not a rate matrix. Up to our knowledge, these are the first examples of generic embeddable matrices (that is with different eigenvalues) satisfying this property (see also Remark \ref{rk:PreviousExamples}). 

\begin{ex}\label{ex:l=-1}

Rounding to the 10-th decimal and taking $l=-1$ the matrix $M$ in the previous example is:
\begin{equation}\label{eq:ExMatrix}
M=\begin{small}
\begin{pmatrix}
	0.1428588867 & 0.3571393697 & 0.3571463443 & 0.1428553993\\
	0.1428588866 & 0.3571411134 & 0.3571446008 & 0.1428553992\\
	0.1428553992 & 0.3571446008 & 0.3571411134 & 0.1428588866\\
	0.1428553993 & 0.3571463443 & 0.3571393697 & 0.1428588867\\
\end{pmatrix}
\end{small}.
\end{equation}

\noindent Further computations show that
\begin{center}
	$
	\begin{small}\Log(M)= \frac{\pi}{4}
	 \begin{pmatrix}
		-17 & 12 & 8 & -3\\
		3 & -13 & 5 & 5	\\
		5 & 5 & -13 & 3 \\
		-3 & 8 & 12 & -17\\		
	\end{pmatrix}
	\quad \textrm{ and }\quad \Ll{-1}{M} = \frac{\pi}{4}	
	\begin{pmatrix}
		-21	& 4 & 16 & 1\\
		7 & -9 & 1 & 1\\
		1 & 1 & -9 & 7 \\
		1 & 16 & 4 & -21\\	
	\end{pmatrix}
	\end{small}
	$.
\end{center}

\noindent  Thus, $M$ is an embeddable matrix whose principal logarithm is not a rate matrix.
\end{ex}

The reader may note that the matrix  $M$ above is actually a SS matrix. Indeed, this happens for all the matrices $M$ in Example \ref{ex:CounterexK}. Next theorem proves that one can perturb the entries of those matrices in order to obtain embeddable matrices  with no symmetry constrains whose only Markov generator is still $\Log_l$.

\begin{thm}\label{thm:OpenSet}
For any $l\in \ZZ$, there is a non-empty Euclidean open set of embeddable Markov matrices  whose unique Markov generator is $\Ll{l}{M}$.
 In particular, there is a non-empty Euclidean open set of $4\times 4$  Markov matrices that are embeddable and whose principal logarithm is not a rate matrix.
\end{thm}

\begin{proof}
Let us define the matrix $$R=\begin{small}
   \begin{pmatrix}
   1 & 0 & 0 & 0\\
   0 & 1 & 0 & 0\\
   0 & 0 & 1 & 1\\
   0 & 0 & i & -i\\
   \end{pmatrix}
   \end{small}
   .$$

Given $\EE = (\delta_1,\dots,\delta_{12}) \in \RR^{12}$ consider the matrix
$$
   A_\EE=\begin{small}
   \begin{pmatrix}
   1 & \delta_1 & \delta_4 & \delta_7\\
   0 & 1 & \delta_5 & \delta_8\\
   0 & \delta_2 & 1 &\delta_9 \\
   0 & \delta_3 & \delta_6 & 1\\
   \end{pmatrix}
   \end{small}.$$

Now, taking $P_+$ and $P_-$ as in Example \ref{ex:CounterexK} we define the matrices $S$ and $D_\EE$ depending on the sign of $l$ as follows:\
\begin{itemize}

	\item If $l\geq 0$ take:
	   $S= P_+ \; R^{-1} =\begin{small}
	   \begin{pmatrix}
		   1 & 6l+2 & 1 & -1\\
		   1 & -2l-1 & 0 & -1\\
		   1 & -2l+1 & 0 & 1\\
		   1 & 6l+2 & -1 & 1\\
	   \end{pmatrix}\end{small}$,

   		$D_\EE= \dd \big(  \ 1 \ , \   (1+\delta_{10}) e^{(1-8l)\pi} \ , \ \delta_{11} + i(1+\delta_{12})e^{-2\pi(1+2l)}  \ , \ \delta_{11} - i(1+\delta_{12})e^{-2\pi(1+2l)}  \ \big)$.

	\item If $l< 0$ take:
		$S= P_- \; R^{-1} = \begin{small}
	   	\begin{pmatrix}
			1 & 6l+1 & 1 & -1\\
		   	1 & -2l & 0 & -1\\
		   	1 & -2l & 0 & 1\\
   			1 & 6l+1 & -1 & 1\\
		\end{pmatrix} \end{small}$,

		$D_\EE= \dd \big(  \ 1 \ , \   (1+\delta_{10}) e^{(1+8l)\pi} \ , \ \delta_{11} + i(1+\delta_{12})e^{-4l\pi}  \ , \ \delta_{11} - i(1+\delta_{12})e^{-4l\pi}  \ \big)$.

\end{itemize}

Set $\kappa\in (0,1)$ small enough so that the matrix $A_\EE$ is invertible if $|\delta_i|<\kappa$, $i=1,\ldots,9$. For such $\delta_i$, $i=1,\ldots,9$, and $\delta_{10},\delta_{11},\delta_{12}\in \RR$  we can define $M_\EE := P_\EE  \; D_\EE \; P_\EE ^{-1}$  where  $P_\EE := S \; A_\EE \; R $.  Note that  $M_0$ is the Markov matrix $M$ in  Example \ref{ex:CounterexK}. In particular, $\Ll{l}{M_0}$ is a rate matrix (and hence $M_0$ is embeddable) while, for $k\neq l$, $\Log_k(M_0)$ is not.

By construction, the first column of $P_{\EE}$ is an eigenvector of $M_\EE$ of eigenvalue 1. A simple computation shows that it is the vector $(1,1,1,1)$ and hence the rows of $M_\EE$ sum to one. If we take $\EE$ in
\begin{eqnarray*}
X:= \big\lbrace \delta=(\delta_1,\dots,\delta_{12}): |\delta_i|<\kappa \text{ for } i=1,\dots,12 \big\rbrace  \subseteq \RR^{12},
\end{eqnarray*}
the second eigenvalue $(1+\delta_{10}) e^{-7\pi}$ of $M_{\EE}$ is positive. On the other hand, the third and fourth eigenvalues and eigenvectors are a  conjugated pair. Since $S$ and $A_\EE$ are real matrices and the third and fourth column-vector of $R$ are a conjugated pair of vectors we deduce that $M_\EE$ is a real matrix whose rows sum to one. By making $\kappa$ smaller if necessary, we can assume that for any $\EE\in X$
the matrix  $M_\delta$ is Markov and non-singular.

 Now, let $\mathcal{M}_{\mathbf{1}}$ be the set of $4\times4$ real matrices  with rows summing to one. Additionally to the map $f:X\rightarrow \mathcal{M}_\mathbf{1},$ $f(\delta)=M_\delta$, let us define the maps $$g,\, h,\, j : f(X)  \longrightarrow M_4(\RR)$$ by
$$ g (M):= \Ll{l-1}{M}, \quad
h (M):= \Ll{l}{M}, \quad
j (M):=  \Ll{l+1}{M}.$$

We write $U\subset M_4(\RR)$ for the open set of matrices with non-zero entries and at least one negative entry outside the diagonal, and we write  $V\subset M_4(\RR)$ for the open set of matrices with non-zero entries and positive off-diagonal entries. As claimed earlier, a straightforward computation shows that $f(0)$ is indeed the matrix $M$ in  Example \ref{ex:CounterexK} and hence $g (f(0)), j (f(0))\in U$ and $h (f(0))\in V$.

Note that $f$, $g $, $h $ and $j $ are continuous on  their respective domains, so $g ^{-1}(U)$, $h ^{-1}(V)$ and $j ^{-1}(U)$ are open sets in $f(X)$ containing $f(0)$. Therefore, $W:= g ^{-1}(U) \cap h ^{-1}(V) \cap j ^{-1}(U)  \subseteq f(X)$ is a non-empty open set in $f(X)$  (it contains $f(0)$). Moreover, for any $M\in W$ we have $M \in h ^{-1}(V)$. Since $exp\circ h = id$, this implies that $M$ is the exponential of a matrix in $V$, that is, $M$ is the exponential of a rate matrix and hence it is a Markov matrix whose $\Log_l$ is a Markov generator. Furthermore,   $\Ll{l-1}{M}$ and $\Ll{l+1}{M}$ are not  rate matrices, because $M$ is included $g^{-1}(U)$ and $M\in j^{-1}(U)$ respectively. As the entries of $\Log_k(M)$ depend linearly on $k$, we get that any matrix in $W$ is an embeddable Markov matrix whose only Markov generator is $\Ll{l}{M}$.

  To conclude the proof we  check that this set contains a non-empty open subset of $\mathcal{M}_{\mathbf{1}}$.
\begin{claim} \label{cl:fInjective}\
Consider the set $Y=\{\EE\in X\mid \delta_6+\delta_9=0\}$. Then,  $f$ is injective in $X\setminus Y$.
\end{claim}

From the claim we have that $f_{|X\setminus Y}:X\setminus Y \rightarrow \mathcal{M}_{\mathbf{1}}$ is injective and hence $f(X\setminus Y)=f(X)\setminus f(Y)$. Moreover, as $X\setminus Y$ is open in $\RR^{12}\simeq \mathcal{M}_{\mathbf{1}}$ and $f$ is continuous on its whole domain we infer, by the invariance of domain theorem, that  $f_{|X\setminus Y}$ is a homeomorphism between $(X\setminus Y)$ and its image, and hence $f(X\setminus Y)$ is an open set of $\mathcal{M}_{\mathbf{1}}$ . To conclude, it is enough to show that the open set $f(X\setminus Y)\cap W$ is not empty. Since $f(Y)$ is an (open set of an) affine algebraic variety of dimension $\leq 11$, the interior of $f(Y)$ is empty and we deduce that $f(0)$ is adherent to $f(X\setminus Y)=f(X)\setminus f(Y)$. In particular, $f(X\setminus Y)$ cuts the neighbourhood $W$ of $f(0)$, and this finishes the proof.
\end{proof}

\begin{proof}[Proof of Claim \ref{cl:fInjective}] \
Since the eigenvalues of any $M_{\delta}$ are all simple, the values of $\delta_{10}, \delta_{11}, \delta_{12}$ are completely determined by $M_{\delta}$.
It remains to see that $M_{\delta}$ also determines the other values of $\delta_i$, $i=1,\ldots,9$ as long as $\delta\in X \setminus Y$.
Let $c_i$ denote the $i$-th column of $A_\EE$.
%
As $P_\EE =  S \; A_\EE \; R$, we have that  the following are eigenvectors of $M_\EE$:
\begin{itemize}
 \item $v_1 = (1,1,1,1)^t$, with eigenvalue $\lambda_1=1$.
 \item $v_2 = S c_2$, with positive eigenvalue $\lambda_2=
 \begin{cases}
 (1+\delta_{10}) e^{(1-8l)\pi} & \text{if } l\geq 0,\\
(1+\delta_{10}) e^{(1+8l)\pi}  & \text{if } l< 0.\\
\end{cases}$
 \item $v_3 = S  (c_3+ i \ c_4 )$, with complex eigenvalue with positive imaginary part $$\lambda_3= \begin{cases}
  \delta_{11} + i(1+\delta_{12})e^{-2\pi(1+2l)}  & \text{if } l\geq 0,\\
\delta_{11} + i(1+\delta_{12})e^{-4l\pi}   & \text{if } l< 0.\\
\end{cases}$$
 \item $ v_4 = S  (c_3- i\  c_4) $, with complex eigenvalue with negative imaginary part $$\lambda_4=\overline{\lambda_3}=\begin{cases}
  \delta_{11} - i(1+\delta_{12})e^{-2\pi(1+2l)}  & \text{if } l\geq 0,\\
\delta_{11} - i(1+\delta_{12})e^{-4l\pi}   & \text{if } l< 0.\\
\end{cases}$$
\end{itemize}

Let  us assume that there are $\EE,\widetilde{\EE}\in X\setminus Y$ so that $M_\EE = M_{\widetilde{\EE}}\,$ and write $\widetilde{v_1},\ \widetilde{v_2},\ \widetilde{v_3},\ \widetilde{v_4}$ for the corresponding eigenvectors of $M_{\widetilde{\delta}}$.  The aim is to show that $\delta = \widetilde{\delta}$.
Using again that the eigenvalues are simple, we have that there are $z_2,z_3,z_4 \in \CC$ such that $v_i = z_i \; \widetilde{v_i}$, $i=2,3,4$.

\begin{itemize}
\item[-] From $v_2 = z_2 \; \widetilde{v_2}$, we have that  $S c_2 =   z_2  S  \widetilde{c_2} = S \left(  z_2  \; \widetilde{c_2} \right)$ and hence we get $ c_2 =  z_2  \; \widetilde{c_2}$. From the second component of $c_2$ and $\widetilde{c_2}$, we deduce that $z_2 = 1$. Hence, $ c_2 =  \widetilde{c_2} $ which implies that $\delta_i = \widetilde{\delta_i}$ for $i=1,2,3$.

\item [-]From $v_3 = z_3 \; \widetilde{v_3}$, we deduce that  $S\;(c_3+ i \ c_4 ) =  z_3 \; S \; (\widetilde{c_3}+ i \ \widetilde{c_4} ) = S \;\left(  z_3  \; (\widetilde{c_3}+ i \ \widetilde{c_4} )  \right)$, and hence $ c_3 + i \ c_4 = z_3  \; (\widetilde{c_3}+ i \ \widetilde{c_4})$. Write  $z_3 = a + b i$ with $a,b\in \RR$. By looking at the third and fourth components,  we obtain that $a = 1 +  b \widetilde{\delta_9}$ and $ a = 1 -  b  \widetilde{\delta_6}$. Since, $(\widetilde{\delta_9}+\widetilde{\delta_6})\neq 0$ this implies that $b=0$ and $a=1$, so $z_3=1$. We derive that $c_3=\widetilde{c_3}$ and $c_4=\widetilde{c_4}$ which implies that $\delta_i = \widetilde{\delta_i}$ for $i=4, \dots ,9$.
\end{itemize}
\end{proof}

\begin{rk}\label{rk:PreviousExamples}
\rm
Examples of {embeddable} Markov matrices for which the principal logarithm is not a rate matrix were already shown in \cite{JJ,K2}. However, the principal logarithm is not clearly defined for the transition matrices in those examples as they have a repeated negative eigenvalue. Moreover, one can consider  the Markov generators given for those examples as directional limits of the principal logarithm of a matrix with a conjugated pair of non-real eigenvalues when the principal argument of this pair tends to $ \pm \pi$.  Indeed, those generators can be written as $$P \; \dd(0,\log(\lambda), \log|\mu|+\pi i ,\log|\mu|-\pi i) \; P^{-1}$$ where $\lambda>0$, $\mu<0$ are eigenvalues of the Markov matrix $M$, and $P$ is a certain matrix that diagonalizes $M$.  The main consequence of this fact is that those examples cannot be extended to obtain an open set of $4\times 4$ Markov matrices whose principal logarithm is not a rate matrix.
\end{rk}

	The results in this section exhibit that for every $l\in \ZZ$ it is possible to construct an embeddable $4\times 4$ Markov matrix $M$  with only one Markov generator precisely given by $\Ll{l}{M}$.  It follows from Theorem 4 by \cite{Cuthbert73} that, excluding the cases $l=0$ and $l=-1$, there is no analogous construction for $3\times3$ Markov matrices.

\section{Constructing the examples} \label{sec:GenerateExamples}

In this section we recover the results in \cite{Birkhauser} in order to show how we obtained the examples in the previous section. To do so, we used SS matrices with a conjugated pair of non real eigenvalues and positive determinant. As claimed in the section \ref{sec:Preliminaries}, any real logarithm of such a matrix is also a SS matrix and is of the form $\Log_k(M)$ defined in (\ref{eq:LogKDef}). We start by providing a parametrization of such logarithms. To this end,  let us consider the algebraic variety
\begin{equation}\label{eq:Variety}
\var = \{(v_1,\dots,v_6)\in \RR^6 \mid v_4^2 - v_5 v_6 =-1/4\}
\end{equation}

\noindent and for any given $v=(v_1,\dots,v_6)\in \RR^6$ and $\theta \in \RR$ define the matrix
\begin{equation}\label{eq:Q}
Q(\theta ,v):=\begin{footnotesize}
	\begin{pmatrix}
		v_1+v_2-v_3-\theta v_4 &  -v_1-v_2+\theta v_5 	& -v_1-v_2-\theta v_5 	& v_1+v_2+v_3+\theta v_4 \\
        -v_1+v_2-\theta v_6 	 & v_1-v_2-v_3+\theta v_4 & v_1-v_2+v_3-\theta v_4  & -v_1+v_2+\theta v_6\\
        -v_1+v_2+\theta v_6 	 & v_1-v_2+v_3-\theta v_4 & v_1-v_2-v_3+\theta v_4  & -v_1+v_2-\theta v_6 \\
        v_1+v_2+v_3+\theta v_4 & -v_1-v_2-\theta v_5  	& -v_1-v_2+\theta v_5     & v_1+v_2-v_3-\theta v_4\\
	\end{pmatrix}.
\end{footnotesize}
\end{equation}

\noindent


\begin{prop}\label{prop:QisLog}
Given $\theta \in (-\pi,\pi)$ and $v \in \mathcal{V}$ it holds that:
\begin{itemize}

\item[i)] $M:=e^{Q(\theta,v)}$ is a SS matrix with rows summing to 1.

\item[ii)] $e^{Q(\theta+ 2\pi k,v)}=M$ for all $k\in \ZZ$.

\item[iii)] If $\theta \neq 0$, then $M$ has two non-real conjugated pair of eigenvalues and $\Log_k(M)=Q(\theta+ 2\pi k,v)$.

\end{itemize}

\end{prop}
\begin{proof}\

\begin{itemize}
\item[i)] As $Q(\theta,v)$ is a SS matrix, so is $M:=e^{Q(\theta,v)}$. Since the rows of $Q(\theta,v)$ sum to $0$ then $(1,1,1,1)$ is an eigenvector with eigenvalue $0$ of $Q(\theta,v)$ and hence $(1,1,1,1)$ is an eigenvector with eigenvalue $e^0=1$ of $M$, which implies that the rows of $M$ sum to $1$.

\item[ii)]Given $v=(v_1,v_2,v_3,v_4,v_5,v_6)$, let $w$ be the vector $(0,0,0,v_4,v_5,v_6)$. Using (\ref{eq:Q}), it is immediate to check that   $Q(\theta+2\pi k,v)=Q(\theta,v)+ Q(2\pi k,w)$. Since $Q(\theta,v)$ and $Q(2\pi k,w)$ are SS matrices they commute and hence
$$e^{Q(\theta+2\pi k,v)}=e^{Q(\theta,v)}\; e^{Q(2\pi k,w)}.$$
Note that as $v \in \mathcal{V}$ so does $w$. Using that $w=(0,0,0,v_4,v_5,v_6)\in \mathcal{V}$, an immediate computation shows that $e^{Q(2\pi k,w)}=Id$ which concludes this part of the proof.

\item[iii)] By the previous statements, we already know that $Q(\theta+ 2\pi k,v)$ is a logarithm of $M$.
A direct computation shows that the spectrum of $Q(\theta+ 2\pi k,v)$ is
\begin{equation*}
\begin{split}
& \left\lbrace 0, 4v_1, -2v_3+ (\theta+ 2\pi k )\sqrt{4(v_4^2-v_6 v_5)}, -2v_3- (\theta+ 2\pi k) \sqrt{4(v_4^2-v_6 v_5)} \right\rbrace\\
 & = \big\lbrace 0, 4v_1, -2v_3+( \theta+ 2\pi k) \; i, -2v_3- ( \theta+ 2\pi k) \; i
 \big\rbrace.
 \end{split}
\end{equation*}
where the last equality is obtained by using that $v \in \mathcal{V}$. As the eigenvalues of $M$ are the exponential of these eigenvalues and $\theta \neq 0$ this shows that $M$ has a conjugated pair of non-real eigenvalues with principal argument $\pm \theta$. Hence, we have that $\Log_k(M)=Q(\theta+ 2\pi k,v)$.
\end{itemize}
\end{proof}

\begin{rk} \label{rk:VisDetermined}\
\rm  Conversely, given a SS Markov matrix $M$ with eigenvalues $1,\lambda,\mu,\overline{\mu}$, with $\lambda\in (0,1)$ and $\mu \in \CC\setminus \RR$ such that $\Arg(\mu)>0$, there is  $v\in \mathcal{V}$ such that $\Log_k(M) = Q(\Arg(\mu)+2\pi k,v)$ for all $k\in \ZZ$. Moreover, the vector $v$ is uniquely determined by the entries of $M$. The proof of this claim is not included here because it is quite technical and the result is not relevant for the goal of this section.
\end{rk}

We denote by $\poly(\theta)$ the set of those $v \in \RR^6$ such that $Q(\theta,v)$ is a rate matrix and by $\poly(\theta)^c$ its complementary. Note that $\poly(\theta)$  is an unbounded convex polyhedral cone because the entries of $Q(\theta,v)$ are linear expressions on the components of $v$, and hence if $Q(\theta,v)$ is a rate matrix so is $Q(\theta,\lambda v)$ for any $\lambda \geq 0$.

\begin{lema}\label{lema}
Given $\theta \in (-\pi,\pi)$ and $k\in \ZZ$, $k\neq0$,  it holds that
 $\poly(\theta)^c \cap \poly(\theta+ 2\pi k)$ has two connected components $\mathcal{C}^{(k)}_1$ and $\mathcal{C}^{(k)}_2$, where $\mathcal{C}^{(k)}_1$ is the set of solutions to the following inequalities:  \begin{equation}\label{eq:C1}
	\begin{matrix}
			v_1+v_2+v_3+\theta v_4 <0,\\
			v_1+v_2+v_3+(\theta + 2\pi k) v_4 \geq0,\\
			v_1-v_2+v_3-(\theta + 2\pi k) v_4 \geq0,\\

		\end{matrix} \qquad  \qquad \begin{matrix}

			-v_1-v_2+(\theta + 2\pi k) v_5 \geq0,\\
			-v_1-v_2-(\theta + 2\pi k) v_5 \geq0,\\
			-v_1+v_2+(\theta + 2\pi k) v_6 \geq0,\\
			-v_1+v_2-(\theta + 2\pi k) v_6 \geq0.\\
		\end{matrix}
 \end{equation}
Moreover,  $(v_1,v_2,v_3,v_4,v_5,v_6) \in \mathcal{C}^{(k)}_1$ if and only if $(v_1,-v_2,v_3,-v_4,v_6,v_5) \in \mathcal{C}^{(k)}_2$.
\end{lema}
\begin{proof}
Since the rows of $Q(\theta,v)$ sum to zero,  $\poly(\theta)$ is the convex polyhedral cone arising from the inequation system $Q(\theta,v)_{i,j}\geq 0$ for all pairs $(i,j)$ with $i\neq j$. Moreover, due to the symmetries of SS matrices the set of inequalities with $i\in\{1,2\}$ and $i\in\{3,4\}$ are the same.

For ease of reading we take $L= Q(\theta,v)$ and $R= Q(\theta + 2\pi k,v)$ and denote their entries by $l_{i,j}$ and $r_{i,j}$ respectively. According to (\ref{eq:Q}), we have that  $r_{1,2}+r_{1,3} = l_{1,2}+l_{1,3}= 2(-v_1-v_2)$, $r_{2,1}+r_{2,4}=l_{2,1}+l_{2,4}=2(-v_1+v_2)$ and    $r_{1,4}+r_{2,3}  = l_{1,4}+l_{2,3} = 2 (v_1+ v_3)$. The off-diagonal entries of $R$ are non-negative because it is a rate matrix and hence $(-v_1-v_2), (-v_1+v_2), (v_1+ v_3) \geq 0$. Since $|\theta| < |\theta + 2\pi k|$ we have that $-v_1-v_2 \pm (\theta + 2\pi k) v_5 \geq 0$ implies $-v_1-v_2 \pm \theta v_5  \geq 0$ thus $l_{1,2}, l_{1,3}\geq 0$. Analogously, we can see that $l_{2,1}, l_{2,4}\geq 0$. Since $L$ is not a rate matrix, then $l_{1,4}<0$ or $l_{2,3}<0$ and we know that  $l_{1,4}+l_{2,3} = 2 (v_1+ v_3) \geq 0$ thus either $l_{1,4}\geq 0, l_{2,3}<0$ or  $l_{2,3}\geq 0, l_{1,4}<0$ showing that  $\poly(\theta )^c \cap \poly(\theta + 2\pi k)$ has two connected components.
From the definition of $Q(\theta,v)$ one can immediately check that given $v=(v_1,v_2,v_3,v_4,v_5,v_6)$ such that the only negative off-diagonal entry of $Q$ and $L$ is $l_{1,4}$ then we get that for $(v_1,-v_2,v_3,-v_4,v_6,v_5)$ the only negative off-diagonal entry of $Q$ and $L$ is $l_{2,3}$.
The linear inequalities system in (\ref{eq:C1}) is the reduced system arising from the assumption that the only negative off-diagonal entry of $Q$ and $L$ is $l_{1,4}$.
\end{proof}

If we allow the first expression in (\ref{eq:C1}) to vanish, a straightforward computation shows that, if $k\neq0$, the solution space is the convex hull of 10 different rays including the ones associated with the vectors\begin{center}
$w_1:=(-|\theta + 2\pi k|,0,|\theta + 2\pi k|,0,1,1), \qquad w_2:=(-|\theta + 2\pi k|,0,|\theta + 2\pi k|,0,-1,1),$
$w_3:=(-|\theta + 2\pi k|,-|\theta + 2\pi k|,|\theta + 2\pi k|,sign(k),2,0).$
\end{center}

Among these, $w_3$ satisfies that $Q(\theta ,w_3)_{1,4}< 0$ and hence the interior of the convex hull of the rays associated with $w_1$, $w_2$ and $w_3$ is included in $\mathcal{C}^{(k)}_1$. In particular, for any $v\in \var$ such that $v=\lambda_1 w_1 + \lambda_2 w_2 +\lambda_3 w_3$ with $\lambda_i>0$ we have that $v\in \var\cap \mathcal{C}^{(k)}_1$. An example of such a vector is $u\in \mathcal{V}$ defined as:
\begin{equation}\label{eq:VectorU}\
u:= (-|\theta + 2\pi k|,-|\theta + 2\pi k|/2,|\theta + 2\pi k|,sign(k)/2,1,1/2)= \frac{w_1}{4}+ \frac{w_2}{4}+\frac{w_3}{2} .
\end{equation}

This shows that $\mathcal{C}^{(k)}_1 \cap\mathcal{V} \neq  \emptyset$. Note that given $v\in \mathcal{C}^{(k)}_1 \cap\mathcal{V}$ it follows from Proposition \ref{prop:QisLog} that $M=e^{Q(\theta,v)}$ is a SS matrix with rows summing to $0$. Indeed, as $v\in \poly(\theta+2\pi k)$ we have that $\Log_k(M)=Q(\theta+2\pi k,v)$ is a rate matrix and hence $M$ is an embeddable Markov matrix. Moreover, we have that $\Log_k(M)=Q(\theta,v)$ is not a rate matrix because $v\in \poly^c(\theta)$, thus we have a constructive method to obtain embeddable SS matrices whose principal logarithm is not a rate matrix.\\

However, the vector $u$ in (\ref{eq:VectorU}) lies in the boundary of $\mathcal{C}^{(k)}_1 \cap\mathcal{V}$ which implies that the Markov generators obtained will have zero entries and hence such matrix can not be deformed to obtain an open set as done in Theorem \ref{thm:OpenSet}. For instance, by taking $\theta =\pi/2$, $k=1$ and $v$ as in (\ref{eq:VectorU}), $v=(-5\pi/2,-5\pi/4,5\pi/2,1/2,1,1/2)$, we get

$$
L = \frac{\pi}{4}
\begin{pmatrix}
	-26  & 17  & 13  & -4\\
	4 & -14  & 4 & 6 \\
	6 & 4 & -14  & 4\\
	- 4 & 13   & 17   & -26  \\
\end{pmatrix}
\qquad \text{ and } \qquad
R= \frac{\pi}{4}
 \begin{pmatrix}
	-30  & 25 & 5  & 0\\
	0 & -10 & 0 & 10 \\
	10  & 0 & -10  & 0\\
	0 & 5  & 25  & -30 \\
\end{pmatrix}.
$$\

\noindent The examples in Section \ref{sec:Obert} where obtained by taking $\theta  =\pi/2$ and  the vector
\begin{equation*}\
v= (-|\theta + 2\pi k|,-|\theta + 2\pi k|/2,|\theta + 2\pi k|,sign(k)/2,1,1/2)- (\pi/4,0,-\pi/2,0,0,0) .
\end{equation*}
This vector lies in the interior of $\mathcal{C}^{(k)}_1 \cap \mathcal{V}$ and hence the rate matrices obtained do not have any null entry (see Example \ref{ex:l=-1}).


\section*{Acknowledgements}
All authors are partially funded by AGAUR Project 2017 SGR-932 and MINECO/FEDER Projects MTM2015-69135 and MDM-2014-0445. J Roca-Lacostena has received also funding from Secretaria d'Universitats i Recerca de la Generalitat de Catalunya (AGAUR 2018FI\_B\_00947) and European Social Funds.

\bibliographystyle{alpha}

\end{document}